\documentclass[11pt]{amsart}
\usepackage{amsfonts}
\usepackage[all]{xy}
\usepackage{mathrsfs}
\usepackage{bbm}
\usepackage{amssymb}
\usepackage{indentfirst}
\usepackage{latexsym,amsfonts,amssymb,amsmath}

\newtheorem{theorem}{Theorem}[section]
\newtheorem{lemma}[theorem]{Lemma}
\newtheorem{corollary}[theorem]{Corollary}

\theoremstyle{definition}
\newtheorem{definition}[theorem]{Definition}
\newtheorem{proposition}[theorem]{Proposition}
\theoremstyle{remark}
\newtheorem{remark}[theorem]{Remark}

\begin{document}

\title
{Quotients with respect to strongly $L$-subgyrogroups }
\author{Ying-Ying Jin}\thanks{}
\address{(Y.-Y. Jin) Department of General Required Courses, Guangzhou Panyu Polytechnic, Guangzhou 511483, P.R. China} \email{yingyjin@163.com; jinyy@gzpyp.edu.cn}

\author{Li-Hong Xie*}\thanks{* The corresponding author.}
\address{(L.-H. Xie) School of Mathematics and Computational Science, Wuyi University, Jiangmen 529020, P.R. China} \email{yunli198282@126.com; xielihong2011@aliyun.com}

\thanks{
This work is supported by the Natural Science Foundation of Guangdong
Province under Grant (Nos. 2020A1515110458, 2021A1515010381), the Innovation Project of Department of Education of Guangdong Province (No. 2022KTSCX145), Scientific research project of Guangzhou Panyu Polytechnic (No. 2022KJ02), and Natural Science Project of Jiangmen City (No. 2021030102570004880).}

\subjclass[2010]{54H11; 22A30; 22A22; 20N05; 54H99}

\keywords{Topological gyrogroup; quotient topology; locally perfect mappings; local pseudocompactness; local
paracompactness}

\begin{abstract}
A topological gyrogroup is a
gyrogroup endowed with a compatible topology such that the multiplication is
jointly continuous and the inverse is continuous.
In this paper, we study the quotient gyrogroups in topological gyrogroups with
respect to strongly $L$-subgyrogroups, and prove that let $(G, \tau,\oplus)$ be a topological gyrogroup and $H$ a closed strongly $L$-subgyrogroup of $G$, then the natural homomorphism $\pi$ from a topological gyrogroup $G$ to its quotient topology on $G/H$ is an open and continuous mapping, and $G/H$ is a homogeneous $T_1$-space.
We also establish that for a locally compact strongly $L$-subgyrogroup $H$ of a topological gyrogroup $G$, the natural quotient mapping $\pi$ of $G$ onto the quotient space $G/H$ is a locally perfect mapping. This leads us to some interesting results on how properties of $G$ depend on the properties of $G/H$. Some classical results in topological groups are generalized.
\end{abstract}

\maketitle

\section{Introduction}
In topological groups, the following general problem was considered by some scholars. Let $G$ be a topological group, $H$
a closed subgroup of $G$, and $G/H$ the quotient space. Suppose that both $H$ and
$G/H$ belong to some nice class of topological spaces. When can we conclude that
$G$ is in the same class? Graev proved that if $H$ and $G/H$ are metrizable, then $G$ is also metrizable \cite{Gr}. Serre established that if $H$ and $G/H$ are locally compact, then the topological group $G$ is also locally compact \cite{Ser}. In 2005, Arhangel'ski established that the following result:

\begin{theorem}\cite[Theorem 1.2]{ARR}
Suppose that $G$ is a topological group, $H$ a locally compact subgroup of $G$, and $\pi : G \rightarrow G/H$ the natural quotient mapping. Then there exists an
open neighbourhood $U$ of the neutral element $e$ such that $\pi(\overline{U})$ is closed in $G/H$
and the restriction of $\pi$ to $\overline{U}$ is a perfect mapping of $\overline{U}$ onto the subspace $\pi(\overline{U}).$
\end{theorem}

Also, this result leads to interesting conclusions on how
certain topological properties of $G$ depend on the properties of $G/H$, when $H$ is
locally compact \cite{ARR}.

In \cite{Ung}, Ungar studies a parametrization of the Lorentz transformation group.
This leads to the formation of gyrogroup theory, a rich subject in mathematics (among others, the interested reader can consult \cite{AbW,Ferr}). Loosely speaking, a gyrogroup (see Definition \ref{Def:gyr}) is a group-like structure in which the associative law fails to satisfy. Recently, as a generalization of topological groups Atiponrat \cite{Atip} defined the concept of topological gyrogroups,
which is a gyrogroup with a topology such that its binary operation is jointly continuous and the operation of taking the inverse is continuous. Some basic properties of topological gyrogroups are studied by Atiponrat \cite{Atip}. Specially, Atiponrat \cite{Atip} discovered that for a topological gyrogroup, $T_0$ and $T_3$ are equivalent. It is worth noting that Cai, Lin and He in \cite{Cai} proved that every Hausdorff first countable topological gyrogroup is metrizable.

As a generalization of topological groups, it is natural to consider the quotients of topological gyrogroups.
In \cite{Bao1, Bao2}, Bao and Lin considered when the admissible $L$-subgyrogroup $H$ of a strongly topological gyrogroup $G$ with some topological properties,
such as locally compact, submetrizable and so on, the natural quotient mapping has some nice topological properties.
The purpose of this paper is to investigate the quotients of topological gyrogroups. Some results in topological groups are improved.

The paper is organized as follows: In Section 2, we mainly introduce the related concepts and conclusions which are required in this article. In Section 3, we introduce the concept of strongly $L$-subgyrogroups and give some topological properties of the quotient spaces with respect to strongly $L$-subgyrogroups of topological gyrogroups are studied. We mainly show that: Let $(G, \tau,\oplus)$ be a topological gyrogroup and $H$ a closed strongly $L$-subgyrogroup of $G$;
then the natural homomorphism $\pi$ from a topological gyrogroup $G$ to its quotient topology on $G/H$ is an open and continuous mapping, and $G/H$ is a homogeneous $T_1$-space(see Theorem \ref{the3.7}). In Section 4, the quotients with respect to locally compact strongly $L$-subgyrogroups of topological groups are studied.
It is shown that if $G$ is a topological gyrogroup, $H$ a
locally compact strongly $L$-subgyrogroup of $G$, and $\pi: G\rightarrow G/H$ is the natural quotient mapping of $G$
onto the quotient space $G/H$, then there exists an open neighbourhood $U$ of the neutral
element $0$ such that $\pi(\overline{U})$ is closed in $G/H$ and the restriction of $\pi$ to $\overline{U}$ is a perfect mapping
of $\overline{U}$ onto the subspace $\pi(\overline{U})$(see Theorem \ref{the4.2}).

All spaces in this paper satisfy the $T_2$-separation axiom. All mappings are continuous and onto.
The reader may consult \cite{Arha, Eng} for unstated notations and terminology.
\section{Definitions and preliminaries}

Let $G$ be a nonempty set, and let $\oplus  : G  \times G \rightarrow G $ be a binary operation on $G $. Then the pair $(G, \oplus)$ is
called a {\it groupoid.}  A function $f$ from a groupoid $(G_1, \oplus_1)$ to a groupoid $(G_2, \oplus_2)$ is said to be
a groupoid homomorphism if $f(x_1\oplus_1 x_2)=f(x_1)\oplus_2 f(x_2)$ for any elements $x_1, x_2 \in G_1$.  In addition, a bijective
groupoid homomorphism from a groupoid $(G, \oplus)$ to itself will be called a groupoid automorphism. We will write $Aut (G, \oplus)$ for the set of all automorphisms of a groupoid $(G, \oplus)$.
\begin{definition}\cite[Definition 2.7]{Ung}\label{Def:gyr}
 Let $(G, \oplus)$ be a nonempty groupoid. We say that $(G, \oplus)$ or just $G$
(when it is clear from the context) is a gyrogroup if the followings hold:
\begin{enumerate}
\item[($G1$)] There is an identity element $e \in G$ such that
$$e\oplus x=x=x\oplus e \text{~~~~~for all~~}x\in G.$$
\item[($G2$)] For each $x \in G $, there exists an {\it inverse element}  $\ominus x \in G$ such that
$$\ominus x\oplus x=e=x\oplus(\ominus x).$$
\item[($G3$)] For any $x, y \in G $, there exists an {\it gyroautomorphism} $\text{gyr}[x, y] \in Aut(G,  \oplus)$ such that
$$x\oplus (y\oplus z)=(x\oplus y)\oplus \text{gyr}[x, y](z)$$ for all $z \in G$;
\item[($G4$)] For any $x, y \in G$, $\text{gyr}[x \oplus y, y]=\text{gyr}[x, y]$.
\end{enumerate}
\end{definition}

\begin{definition}\cite{Suk3}
Let $(G, \oplus)$ be a gyrogroup. A nonempty subset $H$ of $G$ is called a subgyrogroup, denoted
by $H\leq G$, if the following statements hold:
\begin{enumerate}
\item[(1)] The restriction $\oplus|_{H\times H}$ is a binary operation on $H$, i.e. $(H, \oplus|_{H\times H})$ is a groupoid;
\item[(2)] For any $x, y\in H$, the restriction of $\text{gyr}[x, y]$ to $H$, $\text{gyr}[x, y]|_H: H \rightarrow\text{gyr}[x, y](H)$, is a bijective
homomorphism; and
\item[(3)] $(H, \oplus|_{H\times H})$ is a gyrogroup.
\end{enumerate}
\end{definition}

Furthermore, a subgyrogroup $H$ of $G$ is said to be an {\it $L$-subgyrogroup} \cite{Suk3}, denoted by $H\leq_L G$,
if $\text{gyr}[a, h](H) =H$ for all $a\in G$ and $h\in H$.

In this paper, $\text{gyr}[a,b]V$ denotes $\{\text{gyr}[a,b]v: v\in V\}$.

The following Proposition \ref{Pro:gyr} below summarizes some algebraic properties of gyrogroups.

\begin{proposition}\cite{Ung, Ung1}\label{Pro:gyr}
Let $(G,\oplus)$ be a gyrogroup and $a,b,c\in G$. Then
\begin{enumerate}
\item[(1)] $\ominus(\ominus a)=a$ \hfill{Involution of inversion}
\item[(2)] $\ominus a\oplus(a\oplus b)=b$ \hfill{Left cancellation law}
\item[(3)] \text{gyr}$[a,b](c)=\ominus(a\oplus b)\oplus(a\oplus(b\oplus c))$ \hfill{Gyrator identity}
\item[(4)] $\ominus(a\oplus b)=\text{gyr}[a,b](\ominus b\ominus a)$\hfill{\text{cf.~}$(ab)^{-1}=b^{-1}a^{-1}$}
\item[(5)] $(\ominus a\oplus b)\oplus \text{gyr}[\ominus a,b](\ominus b\oplus c)=\ominus a\oplus c$ \hfill{\text{cf.~}$(a^{-1}b)(b^{-1}c)=a^{-1}c$}
\item[(6)] $\text{gyr}[a,b]=\text{gyr}[\ominus b,\ominus a]$ \hfill{Even property}
\item[(7)] $\text{gyr}[a,b]=\text{gyr}^{-1}[b,a], \text{the inverse of gyr}[b,a]$ \hfill{Inversive symmetry}
\end{enumerate}
\end{proposition}

\begin{definition}\cite[Definition 1]{Atip}
A triple $( G, \tau,  \oplus)$ is called a {\it topological gyrogroup} if and only if
\begin{enumerate}
\item[(1)] $( G, \tau)$ is a topological space;
\item[(2)] $( G, \oplus)$ is a gyrogroup;
\item[(3)] The binary operation  $\oplus: G  \times G  \rightarrow G$ is continuous where $G \times G$ is endowed with the product topology
and the operation of taking the inverse $\ominus(\cdot ) : G  \rightarrow G $, i.e. $x \rightarrow \ominus x$, is continuous.
\end{enumerate}
\end{definition}

\begin{definition}\cite[Definition 2.9]{Ung1}
Let $(G,\oplus)$ be a gyrogroup with gyrogroup operation (or,
addition) $\oplus$. The gyrogroup cooperation (or, coaddition) $\boxplus$ is a second
binary operation in $G$ given by the equation
$$(\divideontimes)~~~~a\boxplus b=a\oplus \text{gyr}[a,\ominus b]b$$ for all $a, b\in G$.
The groupoid $(G, \boxplus)$ is called a cogyrogroup, and is said to
be the cogyrogroup associated with the gyrogroup $(G, \oplus)$.

Replacing $b$ by $\ominus b$ in $(\divideontimes)$, along with $(\divideontimes)$ we have the identity
$$a\boxminus b=a\ominus \text{gyr}[a,b]b$$ for all $a, b\in G$, where we use the obvious notation, $a\boxminus b = a\boxplus(\ominus b)$.
\end{definition}

Let $(G,\oplus)$ be a gyrogroup, $x\in G$ and $A,B\subseteq G$.
We write $A\oplus B=\{a\oplus b:a\in A, b\in B\}$, $x\oplus A = \{x\oplus a : a \in A\}$
and $A\oplus x = \{a\oplus x : a \in A\}$.

\begin{definition}\cite{Ung}
Let $(G,\oplus)$ be a gyrogroup, and let $x\in G$. We define the {\it~left~gyrotranslation~} by $x$ to be
the function $L_x: G\rightarrow G$ such that $L_x(y) =x \oplus y$ for any $y \in G$. In addition, the {\it~right~gyrotranslation} by $x$ is defined to be the function $R_x(y)= y\oplus x$ for any $y\in G$.
%Clearly, all left and right gyrotranslations are homeomorphic mappings from $G$ onto itself.
\end{definition}

\begin{proposition}\cite[Proposition 3, Lemma 4 and Corollary 5]{Atip}\label{Pro2.4}
Let $(G,\oplus)$ be a topological gyrogroup, $x\in G$ and $A,B\subseteq G$.
\begin{enumerate}
\item[(1)] The inverse mapping $inv:G\rightarrow G$, where $inv(x)=\ominus x$ for every $x\in G$, is a homeomorphism;
\item[(2)] The left translation $L_{x}: G \rightarrow G$, where $L_{x}(y) = x \oplus y$ for every $y\in G$, is a
homeomorphism;
\item[(3)] The right translation $R_{x}: G \rightarrow G$, where $R_{x}(y) = y\oplus x$ for every $y\in G$, is a
homeomorphism;
\item[(4)] Let $x, y\in G$. Then $\text{gyr}[x,y]:G\rightarrow G$ is a homeomorphism;
\item[(5)] If $A$ is open in $G$, then $x\oplus A$, $A\oplus x$, $A\oplus B$, $B\oplus A$ and $\ominus A$ are all open in $G$.
\end{enumerate}
\end{proposition}

Let $X$ be a topological space and let $A\subseteq X$. We denote the closure $A$ in $X$ by
$\overline{A}$.

\begin{corollary}\label{cor2.8}
Let $G$ be a topological gyrogroup, and $A$ be a subset of $G$. Then $\overline{\ominus A} = \ominus\overline{A}$.
\end{corollary}
\begin{proof}
It is known that for every homeomorphism $f: G \rightarrow G$, $\overline{f(A)}= f(\overline{A})$ is true.
So we can get $\overline{\ominus A} = \ominus\overline{A}$ by Proposition \ref{Pro2.4} (1).
\end{proof}

\begin{corollary}\label{cor2.9}
Let $G$ be a topological gyrogroup, $A$ be a subset of $G$ and $g, h$ be points of $G$. Then $\text{gyr}[g,h]\overline{A}=\overline{\text{gyr}[g,h]A}$.
\end{corollary}
\begin{proof}
It is known that for every homeomorphism $f: G \rightarrow G$, $f(\overline{A})=\overline{f(A)}$ is true.
So we can get $\text{gyr}[g,h]\overline{A}=\overline{\text{gyr}[g,h]A}$ by Proposition \ref{Pro2.4} (4).
\end{proof}

\begin{proposition}\label{the2.9}
If $H$ is a subgyrogroup ($L$-subgyrogroup) of topological gyrogroup $G$, then $\overline{H}$ is a
subgyrogroup ($L$-subgyrogroup) of $G$.
\end{proposition}
\begin{proof}
We show first that, for all $A, B \subseteq G$, we have $\overline{A}\oplus\overline{B}\subseteq \overline{A\oplus B}$.
 By the definition of a topological gyrogroup, the
mapping $op_2$ is continuous $G\times G \rightarrow G$.
 The set $\overline{A}\oplus\overline{B}$ can be represented in the form $op_2(\overline{A}\times\overline{B})$.
  In the product space $G\times G$ we have that
$\overline{A}\times \overline{B} = \overline{A\times B}$, and it follows, by continuity of the mapping $op_2$, that $op_2(\overline{A\times B})
\subseteq \overline{op_2(A\times B)}=\overline{A\oplus B}$.
Since $op_2(\overline{A\times B})=\overline{A}\oplus \overline{B}$,  we have shown that $\overline{A}\oplus \overline{B} \subseteq \overline{A\oplus B}$.
Let $H$ be a subgyrogroup of $G$. Then $H\oplus H = H$ and $H^{-1} = H$.
By the foregoing, we
have $\overline{H}\oplus \overline{H}\subseteq \overline{H\oplus H}$, and thus $\overline{H}\oplus \overline{H}\subseteq \overline{H}$;
this means that the set $H$ is closed under the
operation of $G$. On the other hand, Corollary \ref{cor2.8} shows that $\overline{\ominus H} = \ominus\overline{H}$; from
this it follows that $\ominus \overline{H}= \overline{H}$, and thus the set $H$ is also closed under taking inverses. We
have shown that $H$ is a subgyrogroup of $G$.

We assume that the subgyrogroup $H$ is an $L$-subgyrogroup. Then, for all $a\in G$ and $h\in H$ we have that $\text{gyr}[a, h](H)=H$, and hence, by Corollary \ref{cor2.9}, that $\text{gyr}[a,h]\overline{H}=\overline{\text{gyr}[a,h]H}=\overline{H}$.
As a consequence,
$\overline{H}$ is an $L$-subgyrogroup.
\end{proof}

\begin{proposition}\label{the2.10}
Every open subgyrogroup $H$ of a topological gyrogroup $G$ is
closed in $G$.
\end{proposition}
\begin{proof}
From Proposition \ref{the2.9} it follows that $\overline{H}$ is a subgyrogroup. Since $H$ is an open set containing the identity $0$, one can easily show that $\overline{H}\subseteq H\oplus H=H$, which implies that $H$ is
closed in $G$. In fact, take any $x\notin H$. Then $x\notin H\oplus H$. This implies that $(\ominus H)\oplus x\cap H=\emptyset$. Since $H$ is an open set containing the identity $0$ and $G$ is a topological gyrogroup, $(\ominus H)\oplus x$ is an open set containing $x$. This implies that $x\notin \overline{H}$.
\end{proof}

\begin{proposition}\label{the2.12}
If $H$ is a locally compact subgyrogroup of a topological gyrogroup $G$,
then $H$ is closed in $G$.
\end{proposition}
\begin{proof}
Let $K$ be the closure of $H$ in $G$. Then $K$ is a subgyrogroup of $G$ by Proposition \ref{the2.9}.
Since $H$ is a dense locally compact subspace of $K$, it follows from \cite[Theorem 3.3.9]{Eng}
that $H$ is open in $K$.
However, an open subgyrogroup of a topological gyrogroup is closed by
Proposition \ref{the2.10}. Therefore, $H=K$, that is, $H$ is closed in $G$.
\end{proof}

%%%%%%%%%%%%%%%%%%%%%%%%%%%%%%%%%%%%%%%%%%%%%%%%
%%%%%%%%%%%%%%%%%%%%%%%%%%%%%%%%%%%%%%%%%%%%%%%%
\section{Quotients with respect to strongly $L$-subgyrogroups}
To study the quotients of gyrogroups, Suksumran \cite[Definition 9]{Suk1} defined a normal subgyrogroup in
a similar fashion in groups. Recall that a subgyrogroup $N$ of a gyrogroup $G$ is {\it normal} in $G$, written $N\unlhd G$, if it is the kernel of a gyrogroup homomorphism of $G$. Also, Suksumran obtained the following result in the same paper:
\begin{proposition}\cite[Proposition 38]{Suk1}\label{PP}
Let $G$ be a gyrogroup. If $H$ is a subgyrogroup of $G$ such that
\begin{enumerate}
\item $\text{gyr}[h, a]= \text{id}_{G}$ for all $h\in H, a \in G$;
\item $\text{gyr}[b, a]H\subseteq H$ for all $b, a \in G$;
\item $a\oplus H=H\oplus a$ for all $ a \in G$,
\end{enumerate}
then $H\unlhd G$, that is $H$ is a normal subgyrogroup of $G$.
\end{proposition}

Recall that a subgyrogroup $H$ of a gyrogroup $G$ is a {\it $L$-subgyrogroup} if $\text{gyr}[a, h]H=H$ for all $h\in H, a \in G$. According to \cite[Proposition 24]{Suk1} it follows that the condition ``(2) $\text{gyr}[b, a]H\subseteq H$ for all $b, a \in G$" in Proposition \ref{PP} implies the condition``$\text{gyr}[b, a]H= H$ for all $b, a \in G$". This leads us to define strongly $L$-subgyrogroups as follow:
\begin{definition}\label{def4.1}
A subgyrogroup $H$ of a gyrogroup $G$ is said to
be a {\it strongly $L$-subgyrogroup}\footnote{In \cite[Definition 3.9]{BX} it is called a strongly subgyrogroup.}, denoted by $H\leqslant_{SL} G$, if $\text{gyr}[a, b](H)\subseteq H$ for all $a, b\in G$.
\end{definition}

\begin{remark}
It is easy to see that each normal subgyrogroup is a strongly $L$-subgyrogroup and each strongly $L$-subgyrogroup is an $L$-subgyrogroup.
There exist strongly $L$-subgyrogroups which are not normal subgyrogroups.
Since group $G$ can be considered the special gyrogroup where $\text{gyr~}[a, b]=\text{id}$ for all $a, b\in G$.
Each subgroup $H\subseteq G$ satisfies $\text{gyr~}[a, b](H)= H$ for all $a, b\in G$, and obviously every subgroup doesn't have to be a normal subgroup. However, we don't known whether there is an $L$-subgyrogroup which is not a strongly $L$-subgyrogroup.
\end{remark}

\begin{proposition}\label{Pro4.3}
Let $G$ be a gyrogroup. If $H\leqslant_{SL} G$, then
$a\oplus(b\oplus H)=(a\oplus b)\oplus H$ for all $a, b\in G$.
\end{proposition}
\begin{proof}
Let $a, b\in G$. By Definition \ref{def4.1}, a direct computation gives
 $a\oplus(b\oplus H)=(a\oplus b)\oplus \text{gyr~}[a, b](H)=(a\oplus b)\oplus H$.
\end{proof}

Let $(G, \tau,\oplus)$ be a topological gyrogroup and $H$ a strongly $L$-subgyrogroup of $G$.
 It follows from \cite[Theorem 20]{Suk3} that $G/H =\{a\oplus H: a\in G\}$ is a partition of $G$. We denote
by $\pi$ the mapping $a\mapsto a\oplus H$ from $G$ onto $G/H$. Clearly, for each $a\in G$, we have
we have $\pi^{-1}(\pi(a))= a\oplus H$,
for each $a \in G$. Denote by $\tau(G)$ the topology of $G$. In the set
$G/H$, we define a family $\tau(G/H)$ of subsets as follows:
$$\tau(G/H) = \{O\subseteq G/H : \pi^{-1}(O)\in\tau(G)\}.$$

The following result improves \cite[Theorem 3.13]{BX}.
\begin{theorem}\label{the3.7}
Let $(G, \tau,\oplus)$ be a topological gyrogroup and $H$ a closed strongly $L$-subgyrogroup of $G$.
Then the natural homomorphism
$\pi$ from a topological gyrogroup $G$ to its quotient topology on $G/H$ is an open and continuous mapping, and $G/H$ is a homogeneous $T_1$-space.
\end{theorem}

\begin{proof}
Since a strongly $L$-subgyrogroup of a gyrogroup $G$ is  an $L$-subgyrogroup, $\pi:G\rightarrow G/H$ is open and continuous by \cite[Theorem 3.7]{Bao}.

Let us now prove the homogeneity of $G/H$. For any $a\in G$, define a mapping $h_a$ of
$G/H$ to itself by the rule $h_a(x\oplus H)=a\oplus(x\oplus H)$. Since $a\oplus(x\oplus H)=(a\oplus x)\oplus H\in G/H$
by Proposition \ref{Pro4.3}, this definition is correct. Since $G$ is a gyrogroup, the mapping $h_a$ is evidently a bijection of $G/H$ onto $G/H$. In fact, $h_a$
is a homeomorphism. This can be seen from the following argument.

Take any $x\oplus H\in G/H$ and any open neighbourhood $U$ of $0$. Then $\pi((x\oplus U)\oplus H)$ is a basic
neighbourhood of $x\oplus H$ in $G/H$. Similarly, the set $\pi(a\oplus ((x\oplus U)\oplus H))$ is a basic neighbourhood of
$a\oplus (x\oplus H)$ in $G/H$. Since, obviously, $h_a(\pi((x\oplus U)\oplus H))= \pi(a\oplus ((x\oplus U)\oplus H))$, it easily follows that $h_a$ is a
homeomorphism. Now, for any given $x\oplus H$ and $y\oplus H$ in
$G/H$, we can take $a =y\boxminus x$. Then $h_a(x \oplus H)=a\oplus(x \oplus H)=(a\oplus x) \oplus H=((y\boxminus x)\oplus x)\oplus H=y \oplus H$, by Proposition \ref{Pro4.3}. Hence, the quotient space $G/H$ is
homogeneous. It is a $T_1$-space, since all cosets $x\oplus H$ are closed in $G$ and the mapping
$\pi$ is quotient.
\end{proof}

In \cite[Theorem 3]{Atip}, it was proved that every $T_1$ topological gyrogroup is regular. In fact, this result holds for the quotient space $G/H$ of a topological gyrogroup $G$ with respect to a strongly $L$-subgyrogroup $H$ of $G$.

\begin{lemma}\label{lem12}
Suppose that $(G, \tau,\oplus)$ is a topological gyrogroup, $H$ is a closed strongly $L$-subgyrogroup of $G$, it
is the natural quotient mapping of $G$ onto the quotient space $G/H$, and let $U$ and $V$ be open
neighbourhoods of the neutral element 0 in $G$ such that $\ominus V\oplus V\subseteq U$. Then $\overline{\pi(V)}\subseteq \pi(U)$.
\end{lemma}

\begin{proof}
Take any $x\in G$ such that $\pi(x)\in\overline{\pi(V)}$.
Since $V\oplus x$ is an open neighbourhood
of $x$ and the mapping $\pi$ is open, $\pi(V\oplus x)$ is an open neighbourhood of $\pi(x)$.
Therefore, $\pi(V\oplus x)\cap\pi(V)\neq\emptyset$. It follows that, for some $a\in V$ and $b\in V$, we have $\pi(a\oplus x) = \pi(b)$,
that is, $a\oplus x=b\oplus h$, for some $h\in H$. Hence, $x=\ominus a\oplus(b\oplus h)=(\ominus a\oplus b)\oplus\text{gyr}[\ominus a,b]h=(\ominus a\oplus b)\oplus h_1$,
for some $h_1\in H$. Since $\ominus a\oplus b\in \ominus V\oplus V\subseteq U$, we get $x\in U\oplus H$.
Therefore, $\pi(x)\in \pi(U\oplus H)=\pi(U).$
\end{proof}

\begin{theorem}\label{th4.8}
For any topological gyrogroup $(G, \tau,\oplus)$ and any closed strongly $L$-subgyrogroup $H$ of $G$, the
quotient space $G/ H$ is regular.
\end{theorem}

\begin{proof}
Let $\pi$ be the natural quotient mapping of $G$ onto the quotient space $G/H$, and
let $W$ be an arbitrary open neighbourhood of $\pi(0)$ in $G/H$, where $0$ is the neutral element
of $G.$ By the continuity of $\pi$, we can find an open neighbourhood $U$ of 0 in
$G$ such that $\pi(U)\subseteq W$.
 Since $G$ is a topological gyrogroup, we can choose an open neighbourhood $V$ of 0
such that $\ominus V\oplus V\subseteq U$. Then, by Lemma \ref{lem12}, $\overline{\pi(V)} \subseteq\pi(U)\subseteq W$. Since $\pi(V)$ is an open
neighbourhood of $\pi(0)$, the regularity of $G/H$ at the point $\pi(0)$ is verified. Now it follows
from the homogeneity of $G/H$ that the space $G/H$ is regular.
\end{proof}

Here is an obvious statement which is sometimes quite useful.
\begin{proposition}\label{pro3.8}
Suppose that $(G, \tau,\oplus)$ is a topological gyrogroup, $H$ is a closed strongly $L$-subgyrogroup of
$G$, $\pi$ is the natural quotient mapping of $G$ onto the quotient space $G/H$, $a\in G$, $\lambda_a$
is the left translation of $G$ by $a$ (that is, $\lambda_a(x)=a\oplus x$, for each $x\in G$), and $h_a$ is the left
translation of $G/H$ by $a$ (that is, $h_a(x\oplus H) =a\oplus(x\oplus H)$, for each $x\oplus H\in G/H$). Then $h_a$
is homeomorphisms of $G/H$, and $\pi\circ \lambda_a=h_a\circ\pi$.
\end{proposition}

\begin{theorem}
Suppose that $(G, \tau,\oplus)$ is a topological gyrogroup, $H$ is a closed strongly $L$-subgyrogroup of
$G$, $X$ is a subgyrospace of $G$, $\pi$ is the natural homomorphism of $G$ onto the quotient space
$G/H$, and $Y= \pi(X)$. Suppose also that the space $H$ and the subspace $Y$ of $G/H$ are
first-countable. Then $X$ is also first-countable.
\end{theorem}
\begin{proof}
By Proposition \ref{pro3.8}, we can assume that the neutral element 0 of $G$ is in $X$
and, for the same reason, it suffices to verify that $X$ is first-countable at 0. Let us fix a
sequence of symmetric open neighbourhoods $W_n$ of 0 in $G$ such that $W_{n+1}\oplus W_{n+1}\subseteq W_n$, for each
$n\in\omega$, and $\{W_n \cap H: n\in\omega\}$ is a base for the space $H$ at 0. We also fix a sequence of
open neighbourhoods $U_n$ of 0 in $G$ such that $\{\pi(U_n)\cap Y: n\in\omega\}$ is a base for $Y$ at $\pi(0)$.
Now put $B_{i,j} = W_i\cap U_j \cap X$, for $i, j \in\omega$. To finish the proof, it suffices to establish the
following:

Claim. The family $\eta= \{B_{i,j}: i, j \in\omega\}$ is a base for $X$ at 0.

Clearly, each $B_{i,j}$ is open in $X$ and contains 0. Now take any open neighbourhood $O$
of 0 in $G$. Let us show that some element of $\eta$ is contained in $O$. There exists an open
neighbourhood $V$ of 0 in $G$ such that $V\oplus V\subseteq O$. Choose $m\in \omega$ such that $W_m\cap H\subseteq V$.
Further, there exists $k\in\omega$ to such that
$$\pi(U_k)\cap Y\subseteq\pi(V\cap W_{m+1}).$$
Let us verify that
$B_{m+1,k}\subseteq O.$
Take any $z \in B_{m+1,k} = W_{m+1}\cap U_k \cap X$. Then $z\in U_k \cap X \subseteq (V\cap W_{m+1})\oplus H$, since
$\pi(z)\in \pi(U_k)\cap Y\subseteq \pi(V\cap W_{m+1})$. However, $z$ does not belong to $W_{m+1}\oplus(G\setminus W_m)$,
since $W_{m+1}\oplus W_{m+1}\subseteq W_m$ and $z\in W_{m+1}= \ominus W_{m+1}$. Hence, $z\in(V\cap W_{m+1})\oplus(H\cap W_m)$. Since
$W_m \cap H\subseteq V,$ we conclude that $z \in V\oplus V\subseteq O$. Thus, $B_{m+1,k}\subseteq O$, and $\eta$ is a base for $X$ at
0. Since $\eta$ is countable, it follows that $X$ is first-countable at 0.
\end{proof}
\begin{corollary}\label{cor24}
Suppose that $(G, \tau,\oplus)$ is a topological gyrogroup and $H$
is a closed strongly $L$-subgyrogroup of $G$. If the spaces $H$ and $G/H$ are first-countable, then the space $G$
is also first-countable.
\end{corollary}

\section{Quotients with respect to locally compact subgyrogroups}
In this section, we will establish that for a locally compact strongly $L$-subgyrogroup $H$ in topological gyrogroup $G$,
the natural quotient map $\pi$ of $G$ onto the quotient space $G/H$ has some good properties locally.
This will lead us to some interesting results about how the properties of $G$ depend on the properties of $G/H$ when $H$ is locally compact.

The following result improves \cite[Theorem 3.3]{Bao1}.
\begin{theorem}\label{the4.1}
Suppose that $(G, \tau, \oplus)$ is a topological gyrogroup, $H$ is a closed
strongly $L$-subgyrogroup of $G$, $P$ is a closed symmetric subset of $G$ such that $P$ contains an open
neighbourhood of the neutral element $0$ in $G$, and that $\overline{P\oplus(P\oplus P)}\cap H$ is compact. Let $\pi:G\rightarrow G/H$
be the natural quotient mapping of $G$ onto the quotient space $G/H$. Then the restriction $f$
of $\pi$ to $P$ is a perfect mapping of $P$ onto the subspace $\pi(P)$ of $G/H$.
\end{theorem}

\begin{proof}
Clearly, $f$ is continuous. First, we claim that $f^{-1}(f(a))$ is compact for each $a\in P$. Indeed,
from the definition of $f$, we have $f^{-1}(f(a))=(a\oplus H)\cap P$. Since
the subspace $(a\oplus H) \cap P$ and $H\cap (\ominus a) \oplus P)$
 are homeomorphic, thus both of them are closed in $G$, since $H$ is closed.
 Since $\ominus a\in\ominus P =P$, we have $H\cap ((\ominus a)\oplus P)\subseteq H \cap(P\oplus P)\subseteq \overline{P\oplus(P\oplus P)}\cap H$.
Hence, $H\cap ((\ominus a) \oplus P)$ is compact and so is the set $f^{-1}(f(a))$.

Next we prove that the mapping $f$ is closed. Let us fix any closed subset $M$ of $P$,
and let $a$ be a point of $P$ such that $f(a)\notin f(M)$. Then $f^{-1}(f(a))\cap M=\emptyset$, i.e., $((a\oplus H)\cap P)\cap M=\emptyset$. In particular,
$(a\oplus H)\cap \overline{P\oplus P}\cap M=\emptyset $, since $M\subseteq P$ and $P$ contains $0$. Clearly, the subspace $(a\oplus H)\cap \overline{P\oplus P}$ and
$H\cap \overline{\ominus a\oplus(P\oplus P)}\subseteq H\cap\overline{P\oplus(P\oplus P)}$ are homeomorphic. Thus from the compactness of $H\cap\overline{P\oplus(P\oplus P)}$ it follows that $(a\oplus H)\cap \overline{P\oplus P}$ is compact.
Since $(a\oplus H)\cap \overline{P\oplus P}\cap M=\emptyset$ and $M$ is closed, there exists
an open neighbourhood $W$ of $0$ in $G$ such that $W\oplus ((a\oplus H)\cap \overline{P\oplus P})\cap M=\emptyset $.

Note that the set $\pi(W \oplus a)$ is an open
neighborhood of $f(a)$ in $G/H$, since the quotient mapping $\pi$ is open and $W \oplus a$ is an open neighbourhood of $a$. According to assumption we can take $W\subseteq P$ such that $\ominus W=W$. We claim that $\pi(W \oplus a)\cap f(M)=\emptyset$,
which implies that $f(a)\notin \overline{f(M)}$. That means that $f(M)$ is closed in $f(P)$.

If $\pi(W \oplus a)\cap f(M)\neq\emptyset$, then we have that $\pi^{-1}(\pi(W \oplus a))\cap \pi^{-1}(f(M))\neq\emptyset$, i.e., $(W \oplus a)\oplus H\cap M\oplus H\neq\emptyset$.
Thus there are $y\in W$ and $m\in M$ such that $(y\oplus a)\oplus H=m\oplus H$. Since $H$ is a strongly $L$-subgyrogroup, we have that $y\oplus (a\oplus H)=m\oplus H$ by Proposition \ref{Pro4.3}. Then $a\oplus h=\ominus y\oplus m\in\overline{P\oplus P}$, for some $h\in H$,
since $\ominus y\in \ominus W = W \subseteq P$ and $m\in M\subseteq P$.
Besides, $a\oplus h\in a\oplus H$. Hence,  $a\oplus h\in (a\oplus H)\cap \overline{P\oplus P}$ and
$m=y\oplus(a\oplus h)\in W\oplus ((a \oplus H)\cap \overline{P\oplus P})$. Thus, $m\in M \cap W\oplus ((a \oplus H)\cap \overline{P\oplus P})$,
which contradicts $W\oplus ((a\oplus H)\cap \overline{P\oplus P})\cap M=\emptyset $.
\end{proof}

The following result improves \cite[Theorem 3.4]{Bao1}.

\begin{theorem}\label{the4.2}
Suppose that $(G, \tau, \oplus)$ is a topological gyrogroup, $H$ is a locally compact strongly $L$-subgyrogroup of $G$,
and $\pi: G \rightarrow G/H$ is the natural quotient mapping of $G$
onto the quotient space $G/H$.  Then there
exists an open neighborhood $U$ of the identity element $0$ such that $\pi(\overline{U})$ is closed in $G/H$ and the restriction
of $\pi$ to $\overline{U}$ is a perfect mapping from $\overline{U}$ onto the subspace $\pi(\overline{U})$.
\end{theorem}

\begin{proof}
First, $H$ is closed in $G$ by Proposition \ref{the2.12}.
Since $H$ is locally compact,
there exists an open neighbourhood $V$ of 0 in $G$ such that $\overline{V\cap H}$ is compact.
By the regularity of $G$, we can choose an open
neighborhood $W$ of 0 such that $\overline{W}\subseteq V$. Hence $\overline{W}\cap H$ is compact. Let $U_0$ be an arbitrary symmetric open neighborhood of 0 such that $U_0\oplus(U_0\oplus U_0)\subseteq W$. By the joint continuity,
we have $\overline{U_0}\oplus(\overline{U_0} \oplus\overline{U_0})\subseteq \overline{U_0\oplus(U_0\oplus U_0)}\subseteq \overline{W}$. Then the set $P=\overline{U_0}$ satisfies all restrictions on $P$ in Theorem \ref{the4.1}. It follows from
Theorem \ref{the4.1} that the restriction of $\pi$ to $P$ is a perfect mapping from $P$ onto the subspace $\pi(P)$.

Since $\pi$ is an open mapping, the set $\pi(U_0)$ is open in $G/H$. Since $G/H$ is regular by Theorem \ref{th4.8}, there exists an open neighborhood $V_0$ of $\pi(0)$ in $G/H$ such that $\overline{V_0}\subseteq\pi(U_0)$. Hence $U = \pi^{-1}(V_0)\cap U_0$ is an open neighborhood of 0 contained in $P$ such that the restriction $f$
of $\pi$ to $\overline{U}$ is a perfect mapping from $\overline{U}$ onto the subspace $\pi(\overline{U})$. Furthermore, $\pi(\overline{U})$ is closed in $\pi(P)$, and
$\pi(\overline{U})\subseteq \overline{V_0} \subseteq \pi(U_0)\subseteq \pi(P)$.
Then $\pi(\overline{U})$ is closed in $\overline{V_0}$, so that $\pi(\overline{U})$ is closed in $G/H$.
\end{proof}

Recall that a {\it regular ~closed ~set} in a space is the closure of an open subset of this space.

\begin{corollary}\label{cor4.4}
Assume that $\mathscr{P}$ is a topological property preserved by preimages of spaces under perfect
mappings (in the class of completely regular spaces) and also inherited by regular closed sets. Assume further
that $(G, \tau, \oplus)$ is a topological gyrogroup, $H$ is a locally
compact strongly $L$-subgyrogroup of $G$, and the quotient space $G/H$ has the property $\mathscr{P}$. Then
there exists an open neighborhood $\overline{U}$ of the identity element 0 such that $\overline{U}$ has the property $\mathscr{P}$.
\end{corollary}
\begin{proof}
This is an immediate corollary from Theorem \ref{the4.2}.
\end{proof}

Given a space $X$ and a property $\mathscr{P}$, if each point $x$ of $X$ has an open neighborhood $U(x)$ such that
$\overline{U(x)}$ has $\mathscr{P}$, then we say that $X$ has the property $\mathscr{P}$ $locally$ \cite{Arha}. It is well known that local compactness,
countable compactness, pseudocompactness, paracompactness, the Lindel$\ddot{o}$f property,
$\sigma$-compactness, $\check{C}$ech-completeness, the Hewitt-Nachbin completeness, and the property of being a $k$-space are all inherited by
regular closed sets and preserved by perfect preimages, see \cite[Sections 3.7, 3.10, 3.11]{Eng}.
This observation proves the following statement:

\begin{corollary}\label{cor4.5}
Suppose that $G$ is a topological gyrogroup, and that $H$ is a locally
compact strongly $L$-subgyrogroup of $G$ such that the quotient space $G/H$ has some of the following
properties:
\begin{enumerate}
\item[(1)] $G/H$ is locally compact;
\item[(2)] $G/H$ is locally countably compact;
\item[(3)] $G/H$ is locally pseudocompact;
\item[(4)] $G/H$ is locally $\sigma$-compact;
\item[(5)] $G/H$ is locally paracompact;
\item[(6)] $G/H$ is locally Lindel$\ddot{o}$f;
\item[(7)] $G/H$ is locally $\check{C}$ech-complete; and
\item[(8)] $G/H$ is locally realcompact,
\end{enumerate}
then $G$ also has the same property.
\end{corollary}

\begin{corollary}
Let $(G, \tau, \oplus)$ be a topological gyrogroup,
and let $H$ be a locally compact strongly $L$-subgyrogroup. If the quotient space $G/H$ is a $k$-space,
then $G$ is also a $k$-space.
\end{corollary}
\begin{proof}
This statement follows from Corollary \ref{cor4.4},
since the property of being a $k$-space is invariant under taking perfect preimages and a locally
$k$-space is a $k$-space \cite[Sections 3.3]{Eng}.
\end{proof}

Recall that a space $X$ is said to be {\it zero-dimensional} if it has a base consisting of sets which are both
open and closed in $X$.

The following result improves \cite[Theorem 3.5]{Bao1}.

\begin{theorem}
Suppose that $G$ is a zero-dimensional topological gyrogroup and that $H$ is
a locally compact strongly $L$-subgyrogroup of $G$. Then the quotient space $G/H$ is also zero-dimensional.
\end{theorem}
\begin{proof}
Let $\pi:G\rightarrow G/H$ be the natural quotient mapping of $G$ onto the quotient
space $G/H$. According to Theorem \ref{the4.2}, we can fix an open neighbourhood $U$ of the
neutral element 0 of $G$ such that $\pi(\overline{U})$ is closed in $G$ and the restriction of $\pi$ to $\overline{U}$ is a perfect mapping of $U$ onto the subspace $\pi(\overline{U})$. Take any open neighbourhood $W$ of $\pi(0)$ in $G/H$.
Since the space $G$ is zero-dimensional, we can fix an open and closed neighbourhood $V$ of
0 such that $V\subseteq U\cap\pi^{-1}(W)$. Then $\pi(V)$ is an open subset of $G/H$, since the mapping
$\pi$ is open. On the other hand, $\pi(V)$ is closed in $G/H$, since the restriction of $\pi$ to $\overline{U}$
is a closed mapping and $\pi(\overline{U})$ is closed in $G/H$. Clearly, $\pi(V)\subseteq W$. Hence, $G/H$ is
zero-dimensional.
\end{proof}

We recall that the {\it tightness} of a space $X$ is the minimal cardinal $\tau\geq \omega$ with the property
that for every point $x\in X$ and every set $P\subseteq X$ with $x\in \overline{P}$, there exists a subset $Q$ of $P$ such
that $|Q|\leq\tau$ and $x\in\overline{Q}$.

\begin{lemma}\cite[Proposition 4.7.16]{Arha}\label{lem4.7}
Suppose that $f: X\rightarrow Y$ is a closed continuous mapping of a regular
space $X$ onto a space $Y$ of countable tightness. Suppose further that the tightness of every fiber $f^{-1}(y)$, for
$y\in Y$ , is countable. Then the tightness of $X$ is also countable.
\end{lemma}

The following result improves \cite[Theorem 3.10]{Bao1}.
\begin{theorem}
Suppose that $G$ is a topological gyrogroup, and that $H$ is a locally compact
metrizable strongly $L$-subgyrogroup of $G$ such that tightness of the quotient space $G/H$ is countable. Then
the tightness of $G$ is also countable.
\end{theorem}
\begin{proof}
It follows from Theorem \ref{the4.2} that there exists an open neighbourhood $U$ of
the neutral element 0 in $G$ such that $\overline{U}$ is a preimage of a space of countable tightness under
a perfect mapping with metrizable fibers.
Then, by Lemma \ref{lem4.7},
the tightness of $\overline{U}$ is also countable. Since $U$ is a non-empty open subset of the homogeneous space $G$, the
tightness of $G$ is countable.
\end{proof}

%%%%%%%%%%%%%%%%%%%%%%%%%%%%%%%%%%%%%%%%%%%%%%%%
%%%%%%%%%%%%%%%%%%%%%%%%%%%%%%%%%%%%%%%%%%%%%%%%
\section{Acknowledgements}
We wish to thank the reviewers for many valuable advices, and all their
efforts in order to improve the paper.

%%%%%%%%%%%%%%%%%%%%%%%%%%%%%%%%%%%%%%%%%%%%%%%%
%%%%%%%%%%%%%%%%%%%%%%%%%%%%%%%%%%%%%%%%%%%%%%%%
%\bibliographystyle{amsplain}

\end{document}